\documentclass{amsart}

\usepackage{
  bm,
  bbm,
  color,
  a4wide,
  amssymb,
  amsmath,
  latexsym,
  calligra,
  mathrsfs,
  stmaryrd,
  hyperref,
  graphicx
}

\usepackage[2cell,arrow,curve,matrix]{xy}

\newcommand\fCat{\mathfrak{Cat}}

\newcommand\fMon{\mathfrak{Mon}}
\newcommand\fLMon{\mathfrak{Lax}\fMon}
\newcommand\fPMon{\mathfrak{Ps}\fMon}
\newcommand\fFib{\mathfrak{Fib}}
\newcommand\fPFib{\mathfrak{P}\fFib}
\newcommand\MonExt{\fMon\mathfrak{Ext}}

\newcommand\sx{\mathsf{x}}
\newcommand\sC{\mathsf{C}}
\newcommand\Pcar{\mathsf{Pcar}}
\newcommand\Car{\mathsf{Car}}
\newcommand\Groth{\mathsf{Groth}}
\newcommand\Mon{\mathsf{Mon}}
\newcommand\Cat{\mathsf{Cat}}

\newcommand\pt{\bullet}
\newcommand\Clev{\mathsf{Clev}}

\newcommand\Maps{\mathsf{Maps}}

\newcommand\id{\mathsf{Id}}
\renewcommand\ker{\mathsf{Ker}}
\newcommand\aut{\mathsf{Aut}}
\newcommand\End{\mathsf{End}} 
\newcommand\monext{\Mon\mathsf{Ext}} 
\newcommand\regmonext{\mathsf{Reg}\Mon\mathsf{Ext}} 

\newcommand\N{\mathbb N}
\newcommand\cyc{{\sf Z}^2}

\newcommand\To\Rightarrow
\newcommand\xto[1]{\xrightarrow[]{#1}}
\newcommand\ptM{
  \rotatebox[origin=c]{180}{$\bm{\circlearrowright}\hspace{-2.5mm}^\bullet$}_M
}

\newtheorem{Core}{Definition}[section]
\newtheorem{Aux}{Remark}[section]

\newtheorem{De}[Core]{Definition}
\newtheorem{Th}[Core]{Theorem}
\newtheorem{Pro}[Core]{Proposition}
\newtheorem{Le}[Core]{Lemma}
\newtheorem{Cor}[Core]{Corollary}

\newtheorem{Rem}[Aux]{Remark}

\newtheorem{Ex}[Aux]{Example}

\title{Grothendieck's theory of fibred categories for monoids}
\author{Ilia Pirashvili}

\begin{document}


\begin{abstract}
  Grothendieck's theory of fibred categories establishes an equivalence between fibred
  categories and pseudo functors. It plays a major role in algebraic geometry
  \cite{sga1}, \cite{vistoli} and categorical logic \cite{catlog}. This paper aims to
  show that fibrations are also very important in monoid theory. Among other things, we
  generalise Grothendieck's result slightly and show that there exists an equivalence
  between prefibrations (also known as Schreier extension in the monoidal world) and lax
  functors. We also construct two exact sequences which involve various automorphism
  groups arising from a given fibration. This exact sequence was previously only known
  for group extensions \cite{yadav}.
\end{abstract}

\maketitle

\section{Introduction}

  In 1961, Grothendieck introduced fibred categories \cite{sga1} and discovered an
  equivalence between fibred categories and pseudo functors. 
  The theory of fibred categories has many important applications in algebraic
  geometry, begin a core construction that appears in the theory of stacks
  \cite{vistoli}. It also plays an important role in categorical logic \cite{catlog}.

  This paper aims to show that fibrations are also very important in monoid theory. This
  is done in the natural way by considering monoids as one object categories and
  monoid homomorphism as functors. In this vein, we call a monoid homomorphism a
  pre-fibration if the corresponding functors is a pre-fibration.

  It is known \cite{manuell}, though perhaps not well-known, that the notion of a
  prefibration corresponds to a Schreier extension in terms of monoid homomorphisms. A
  fibration in this correspondence is a regular Schreier extension. Indeed, this
  fundamental observation is not mentioned even in quite recent works on Schreier
  extensions of monoids, see \cite{Faul}. Moreover, even where (regular) Schreier
  extensions are discussed, the basic classification theory of Schreier extensions is
  usually stated in a much weaker form than Grothendieck's original result \cite{schP2},
  \cite{PaForumMath}. It is one of the aims of this paper to alleviate this.

  Schreier extensions play a more prominent role in monoid theory than regular Schreier
  extensions \cite{redei, I, Str, bourn, schP1, schP2, APachII}. It is subsequently one
  of the main goals of this paper to generalise Grothendieck's theory to prefibrations
  and lax functors. Though we will restrict ourselves to one-object categories (to stick
  with the theme of accessibility for these interested in monoids and not category
  theory), the generalisation discussed here should also hold in the general case with
  fairly little modification. This is achieved through generalising the semi-direct
  product to a sort of 2-semi-direct product (referred to as $\Groth$ in this paper). 

  Another goal was the study of the automorphism group of pre-fibrations and the
  construction of short exact sequence involving various automorphism groups. Our
  results have previously been obtained in \cite{yadav} for the case of groups.

  As already mentioned, the fact that Schreier extensions are a specialisation of
  Grothendieck's pre-fibrations was observed only fairly recently and is perhaps not so
  well known, yet. This is perhaps due to the rather involved nature of the general
  definition of Grothendieck's pre-fibrations in standard references \cite{borc, catlog,
  vistoli}, which requires a certain effort to learn a fairly substantial portion from
  geometry or logic/category theory, with no immediate link to monoids. This paper aims
  to alleviate this to some degree by presenting categorical constructions with a sole
  focus on monoids.

  The paper is organised as follows: We introduce lax- and pseudo actions of a monoid
  $M$ on a monoid $A$ in Section \ref{sec.2cat_mon}. In such a case, $A$ will be
  referred to as a lax- (resp. pseudo) $M$-monoid. This is just a specialisation of a
  much more general notion, that of a lax- (resp. pseudo) functor acting on a
  2-category. By definition, a lax $M$-monoid is a triple $(A, \phi, \gamma)$ where $A$
  is a monoid while $\phi: A\times M\to A$ and $\gamma: M\times M\to A$ are maps
  satisfying certain properties. The map $\phi$ defines a lax action of $M$ on $A$,
  while $\gamma$ is a sort of $2$-cocycle measuring the non-associativity of the action.
  A lax $M$-monoid is a pseudo $M$-monoid if $\gamma$ takes values in $A^\sx$, where
  $A^\sx$ denotes the group of invertible elements of $A$.

  Section \ref{sec.Groth_th}, which can be read independently, represents the origin of
  this paper. Here, we introduce pre-fibrations of monoids as the specialisations of the
  much more general notions \cite{borc} of pre-fibrations of categories. In more detail,
  for a monoid homomorphism $\sigma: M\to N$, we define the two subsets: $\Pcar(\sigma)$
  and $\Car(\sigma)$ of $M$, known as the sets of $\sigma$-precartesian and
  $\sigma$-cartesian elements, respectively. One has inclusions of sets
  \[A^\sx \subseteq \Car(\sigma) \subseteq \Pcar(\sigma) \subseteq M,\]
  where $A=\ker(\sigma)$. Moreover, $\Car(\sigma)$ is a submonoid of $M$ and
  $\Pcar(\sigma)$ is a right $\Car(\sigma)$-subset of $M$. The homomorphism $\sigma$ is
  called a fibration (resp. prefibration) if the restriction of $\sigma$ on
  $\Car(\sigma)$ (resp. $\Pcar(\sigma)$) is surjective. Several properties of
  pre-fibrations are subsequently stated.

  We proceed to construct a monoid $\Groth(A, \phi, \gamma)$ in Section
  \ref{sec.Groth_const} starting from a lax $M$-monoid $(A, \phi, \gamma)$, through a
  process known as the Grothendieck construction. We prove that the natural projection
  $\Groth(A,\phi,\gamma)\to M$ is a prefibration. Moreover, it is a fibration if $(A,
  \phi, \gamma)$ is a pseudo $M$-monoid.
  
  We construct a lax $N$-monoid structure on $A=\ker(\sigma)$ starting with a
  prefibration $\sigma: M\to N$ in Section \ref{sec.cleave}. This depends on certain
  sections of $\sigma$ known as cleavages. This action is a pseudo action for
  fibrations.

  The main result of Grothendieck's theory, restricted to the setting of monoids, states
  that the Grothendieck construction defines an equivalence between the 2-category of
  pseudo (resp. lax) $N$-monoids and the 2-category of fibrations (resp. prefibrations)
  over $N$.

  The results of Section \ref{sec.cleave} are intimately related to the theory of
  Schreier extensions of monoids \cite{redei}. Some of the facts proved here are also
  contained in \cite{schP1}, \cite{schP2}, \cite{APachII}.

  Section \ref{sec.autfib} studies the group of automorphism of fibrations, while Section
  \ref{sec.commutative} focuses on the case when the commutative kernel of a fibration is a
  commuttaive monoid. We show that several results obtained in \cite{APachII} follow
  immediately from the the general theory of Grothendiek. The section also discusses an
  exact sequence involving various automorphism groups, see Theorem \ref{th.main_7}. This
  exact sequence was previously known only for group extensions \cite{yadav}. Our proof
  is based on some technical results obtained in Section \ref{sec.autfib}. 

  Here, I would like to take the opportunity to thank Graham Manuell for pointing out
  that the link between Schreier extensions and the Grothendieck construction was
  already known.

\section{2-categories relating to monoids}\label{sec.2cat_mon}

  When working with monoids, we are usually dealing with the category $\Mon$. Objects of
  this category are monoids and morphisms are monoid homomorphisms. For the purposes of
  this paper, it will be important to see the category $\Mon$ as a full subcategory of
  the category $\Cat$ of small categories. Recall that under this embedding, a monoid
  $M$ is identified with the one object category $\ptM$. The singular object here is
  denoted by $\pt$. The elements of $M$ are considered as endomorphisms of $\pt$. The
  composition of endomorphisms in $\ptM$ is the multiplication law of $M$.

  In contrast to $\Cat$, we can consider the 2-category of categories $\fCat$. The later
  has a much richer structure, as it maintains the information of natural
  transformations between functors. These are the so called $2$-morphisms or 2-cells of
  $\fCat$. The category $\Cat$ can be seen as the underlying category of $\fCat$, much
  as the objects of a small category $\sC$ can be regarded as is its underlying set.

  This suggest that we take a look at the 2-category $\fMon$, whose underlying category
  is $\Mon$. The 2-morphisms in $\fMon$, meaning the morphisms $f\To g$ between the
  homomorphisms $f, g: M\to N$, are simply elements $n\in N$ such that $nf(m)=g(m)n$ for
  all $m\in M$. This is the same as a natural transformation if one looks at $f$ and $g$
  as functors between one object categories.

  The philosophy of 2-categories suggest that the various notions of category theory are
  just "shadows" of richer structures. This brings us to the notions of lax and
  pseudo-functors, which are 2-categorical generalisations of functors. This viewpoint
  goes back to Grothendieck, see \cite{sga1}. Our aim is to apply a result of
  Grothendieck to monoids and elaborate upon it in this setting. This original result
  established a type of equivalence between fibrations and pseudofunctors. In the course
  of doing this task, we prove an even more general result, by constructing an
  equivalence between prefibrations and laxfunctors. We also give several applications
  in the cohomology theory of monoids. 
  
\subsection{Lax and pseudo actions}
  
  A contravariant functor $\ptM\to\Mon$ is essentially a monoid $A$ with a right
  $M$-action. Specifically, this is given by $\psi: A\times M\to A$ with
  $(a,m)\mapsto\psi_m(a)$ such that
  \begin{eqnarray*}
    \psi_1(a)=a, & & \psi_{mn}(a)=\psi_n(\psi_m(a)) \\
    \psi_m(1)=1, & & \psi_m(ab)=\psi_m(a)\psi_m(b).
  \end{eqnarray*}
  Contravariant lax and pseudo functors $\ptM\to\fMon$ can thusly be considered as
  2-categorical versions of $M$-monoids. We call them lax and pseudo actions of $M$ on a
  monoid $A$. We also say that $A$ is a \emph{lax $M$-monoid}, respectively a
  \emph{pseudo $M$-monoid}. Let us state these notions in terms of monoids:

  \begin{De}\label{de.lax_act}
    A \emph{lax action} of a monoid $M$ on a monoid $A$ is given by two maps
    \[\phi: A\times M\to A,\ (a,m)\mapsto\phi_m(a)\]
    and
    \[\gamma: M\times M\to A,\ (m,n)\mapsto\gamma_{m,n}\]
    such that the following conditions hold:
    \begin{itemize}
      \item [i)] We have $\phi_1(a)=a$, $a\in A$.
      \item [ii)] Let $m,n\in M$ and $a\in A$. One has
        \[\gamma_{m,n}\phi_n(\phi_m(a))=\phi_{mn}(a)\gamma_{m,n}.\]
      \item [iii)] For any elements $m, n, k\in M$, we have
        \[\gamma_{mn,k}\phi_k(\gamma_{m,n})=\gamma_{m,nk}\gamma_{n,k}.\]
      \item[iv)] One has $\gamma_{1,m}=1=\gamma_{m,1}$ for any $m\in M$.
      \item [v)] We have $\phi_m(1)=1$ and $\phi_m(ab)=\phi_m(a)\phi_m(b)$.
    \end{itemize}
    A lax action is called a \emph{pseudo action} if $\gamma_{mn}$ is invertible for all
    $m, n\in M$.
  \end{De}
  
  \begin{Rem}
    In the general theory of lax (resp. pseudo) functors, one requires the existence of
    an element $b\in A$ (resp. $b\in A^\sx$) such that $\phi_1(a)b=ba$ for all $a\in
    A$ instead of $\phi_1=\id$ as we assumed above. This simpliefied condition, however,
    is sufficient for our purposes.
  \end{Rem}

  We will, for simplicities sake, say that $A$ is a lax $M$-monoid if $(A,\phi,\gamma)$
  is a lax action of $M$ on $A$. Likewise for pseudo actions. We see that $M$-monoids
  are exactly lax $M$-monoids for which $\gamma_{m,n}=1$ for all $m,n\in M$.

  \begin{Rem}\label{rem.grp_ext}
    Observe that for groups, the above conditions are familiar equalities in group
    extension theory, see Equalities (8.5) and (8.6) in \cite[ChIV, Section
    8]{homology}, which are obtained as follows: Let
    \[1\to G\xto{\kappa} B \xto{\pi}\Pi \to 1\]
    be a short exact sequence of groups. Choose a map $u:\Pi\to B$ such that $\pi\circ
    u=\id_\Pi$ and consider $f:\Pi\times\Pi\to G$ and $\phi:G\times\Pi\to G$ which are
    uniquely defined by the equations
    \[u(x)g=\phi(g,x)u(x),\quad u(x)u(y)=f(x,y)u(xy).\]
    The pair $(\phi,f)$ is a pseudo action of $\Pi$ on $G$, as shown in the computations
    on page 125 of \cite{homology}.
  \end{Rem}
  
  We also need morphisms between lax and pseudo $M$-monoids. The following is a
  specialisation of the general definition.
  
  \begin{De}
    Consider lax $M$-monoids $(A,\phi,\gamma)$ and $(A',\phi',\gamma')$. A \emph{lax
    $M$-homomorphism} $(A,\phi,\gamma)\to(A',\phi',\gamma')$ is given by a pair
    $(\alpha,\tau)$, where $\alpha:A\to A'$ is a monoid homomorphism and $\tau:M\to A'$,
    $m\mapsto \tau_m$ is a map such that $\tau_1=1$. Moreover, the following two
    identities
    \[\phi'_m(\alpha(a))\tau_m=\tau_m\alpha(\phi_m(a))\]
    and
    \[\gamma'_{m,n}\phi'_n(\tau_m)\tau_n=\tau_{mn}\alpha(\gamma_{m,n})\]
    must hold.
  
    If $\tau$ is additionally invertible, $(\alpha,\tau)$ is called a pseudo
    $M$-homomorphism.
  \end{De}
  
  For a pair of lax M-homomorphisms
  \[(A,\phi,\gamma)\xto{(\alpha,\tau)}(A',\phi',\gamma')
    \xto{(\alpha',\tau')}(A'',\phi'',\gamma'')\]
  of lax $M$-monoids, the composite morphism $(A,\phi,\gamma)\to(A'',\phi'', \gamma'')$
  is given by
  \[(\alpha',\tau')(\alpha,\tau)=(\alpha'\alpha,m\mapsto\tau_m'\alpha'(\tau_m)).\]
  
  \begin{De}
    If $(\alpha,\tau)$ and $(\beta,\theta)$ are two lax morphisms between
    $(A,\phi,\gamma) \to(A',\phi',\gamma')$, a lax 2-\emph{cell} $(\alpha,\tau)\To
    (\beta,\theta)$ is given by an element $c'\in A'$ for which conditions
    \[c'\alpha(a)=\beta(a)c'\quad{\rm and }\quad \phi_m'(c')\tau_m=\theta_mc'\]
    hold. We say that $c'$ is a pseudo 2-cell if $c'$ is invertible.
  \end{De}
  
  One obtains the 2-category of lax $M$-monoids in this fashion, which is denoted by
  $\fLMon_M$. We also denote by $\fPMon_M$ the full sub-2-category of pseudo
  $M$-monoids, pseudo morphisms and pseudo 2-cells.
  
\section{Grothendieck's theory of fibre categories applied to monoids}
\label{sec.Groth_th}
  
  As the title indicates, this section will deal with one object categories, meaning
  essentially monoids, and how the theory of Grothendieck's fibrations looks like in
  this setting. The goal is to formulate all this in terms of monoids.
  
\subsection{Precartesian elements}\label{sec.precart}
  
  Let $\sigma:M\to N$ be a homomorphism of multiplicatively written monoids. As usual,
  $\ker(\sigma)$ denotes the preimage of $1\in N$.
  
  The following is a special case of \cite[Definition 8.1.5]{borc}.
  
  \begin{De}\label{de.sig_precar}
    An element $x\in M$ is \emph{$\sigma$-precartesian} if for a given $z\in M$ such
    that $\sigma(z)=\sigma(x)$, there exists a unique element $y\in \ker(\sigma)$ such
    that $z=xy$.
  \end{De}

  In the notation of Definition \ref{de.sig_precar}, we say that $x$ is
  $\sigma$-precartesian above $n$, where $n=\sigma(x)$. It follows that if $x$ is
  $\sigma$-precartesian and $xu=xv$ for $u,v\in\ker(\sigma)$, then $u=v$. We will refer
  to this as the \emph{weak cancellation property}.

  For a homomorphism $\sigma:M\to N$, let $\Pcar(\sigma)$ be the set of all
  $\sigma$-precartesian elements. It is obvious that $1\in\Pcar(\sigma)$.

  \begin{Le}\label{le.sig_pre_inv}
    Let $\sigma:M\to N$ be a homomorphism of monoids. Assume $x$ and $y$ are
    $\sigma$-precartesian elements of $M$. If $\sigma(x)=\sigma(y)$, there exists a
    unique element $h\in\ker(\sigma)$ such that $y=xh$. Any such $h$ must be invertible.
  \end{Le}

  \begin{proof}
    The existence of such an $h$ follows from Definition \ref{de.sig_precar}. For the
    rest, observe that by symmetry, there must also be a uniquely defined
    $t\in\ker(\sigma)$ for which $x=yt$ holds. Thus, $x=xht$ and $y=yth$. The uniqueness
    assumption of Definition \ref{de.sig_precar} gives now gives us $ht=1=th$. The
    result follows.
  \end{proof}

  \begin{De}
    A homomorphism $\sigma:M\to N$ is called a \emph{prefibration} if the restricted map
    $\Pcar(\sigma)\to N$ is surjective. In other words, if there exists a
    $\sigma$-precartesian element in $M$ above any element of $N$.
  \end{De}

\subsection{Grothendieck's fibrations for monoids}
	
  We will follow \cite{borc} closely, starting with a specialisation of \cite[Definition
  8.1.2]{borc}.

  \begin{De}\label{de.cart_sigma}
    Let $\sigma:M\to N$ be a homomorphism of monoids. An element $x\in M$ is called
    $\sigma$-\emph{cartesian} if the following conditions hold:
    \begin{itemize}
      \item [i)] For a given $z\in M$ and a factorisation $\sigma(z)=\sigma(x)v$, $v\in
        N$, there exists a $y\in M$ such that $z=xy$ and $\sigma(y)=v$.
      \item [ii)] If $y_1,y_2\in M$ are elements such that $xy_1=xy_2$ and
        $\sigma(y_1)=\sigma(y_2)$, then $y_1=y_2$.
    \end{itemize}
    If $x$ is $\sigma$-cartesian and $n=\sigma(x)\in N$, we say that $x$ is
    $\sigma$-cartesian above $n$.
  \end{De}
	
  Observe that in the above Definition, ii) basically says that $y$ in i) is unique. We
  let $\Car(\sigma)$ be the set of all $\sigma$-cartesian elements. Clearly, $1\in
  \Car(\sigma)$.

  \begin{Le}\label{le.sig_cart=sig_precart}
    Let $\sigma:M\to N$ be a homomorphism of monoids. Any $\sigma$-cartesian element is
    $\sigma$-precartesian. In other words, $\Car(\sigma)\subseteq\Pcar(\sigma)$.
  \end{Le}

  Note that the converse is not true in general, see Example \ref{ex.N_to_cycl}.

  \begin{proof}
    Assume $\sigma(x)=\sigma(y)$. Take $z=y$ and $v=1$ in i) of Definition
    \ref{de.cart_sigma} to obtain $y=xh$ for some $h\in \ker(\sigma)$. If $y=xh'$ for
    $h'\in \ker(f)$, we can use Condition ii) of Definition \ref{de.cart_sigma} to
    obtain $h=h'$.
  \end{proof}

  With this in mind, Lemma \ref{le.sig_pre_inv} implies the following:
	
  \begin{Cor}\label{cor.sig_cart_uniq_h}
    Let $\sigma:M\to N$ be a homomorphism of monoids and $x$ and $y$ $\sigma$-cartesian
    elements of $M$ for which $\sigma(x)=\sigma(y)$. There exists a unique element
    $h\in\ker(\sigma)$ for which $y=xh$. Such an $h$ must moreover be invertible.
  \end{Cor}

  \begin{Le}\label{le.inv_sig_cart}
    Let $\sigma:M\to N$ be a monoid homomorphism and $x\in M$. The following
    properties hold:
    \begin{itemize}
      \item [i)] If $x$ is invertible, then $x$ is $\sigma$-cartesian.
      \item [ii)] If $\sigma(x)$ is invertible, then $x$ is $\sigma$-cartesian if
        and only if $x$ is invertible.
    \end{itemize}
  \end{Le}

  \begin{proof}
    i) Take $z$ and $v$ as in Definition \ref{de.cart_sigma}. It follows that
    $v=\sigma(x^{-1}z)$ and $y=x^{-1}z$.

    ii) Assume $x$ is $\sigma$-cartesian. Take $z=1$ and $v=\sigma(x)^{-1}$ in
    Definition \ref{de.cart_sigma} to obtain that $x$ is invertible. The rest is a
    consequence of part i).
  \end{proof}
    
  \begin{Cor}
    Let $\sigma:M\to N$ be a homomorphism. We have $M^\sx\subset\Car(\sigma)$, where
    $M^\sx$ denotes the group of invertible elements of $M$.
  \end{Cor}
    
  The requirement on $x$ to be $\sigma$-cartesian in Lemma \ref{le.inv_sig_cart} ii) can
  not be replaced by it being $\sigma$-precartesian, as Example \ref{ex.N_to_cycl} below
  shows.
    
  \begin{De}
    A homomorphism $\sigma:M\to N$ is called a \emph{fibration} if the restricted
    map $\Car(\sigma)\to N$ is surjective. This can be equivalently expressed as the
    existance of a $\sigma$-cartesian element above every element of $N$.
  \end{De}

  A homomorphism $f:M\to N$ is called an \emph{op-fibration} if $f$ is a fibration when
  considered as a homomorphism of opposite monoids. Grothendieck used the word
  \emph{cofibration} instead if op-fibration, see \cite{sga1}. Due to a clash of
  terminologies from homotopy theory, the term op-fibration is more commonly used
  nowadays.

  When restricted to the monoid world, the terms right- and left fibration are,
  arguably, preferable. These are naturally coming from right- and left monoid actions.
  What we called a fibration would then be called a right fibration in out setting. The
  op-fibration would, in turn, be the left fibration. Likewise for the terms cartesian
  and precartesian.

  Any group epimorphism is both a fibration and co-fibration by Lemma
  \ref{le.inv_sig_cart}. By Corollary \ref{cor.sig_cart_uniq_h} and Section
  \ref{sec.precart}, any fibration is a prefibration. The converse is not true, as the
  following easy example shows.

  \begin{Ex}\label{ex.N_to_cycl}
    The homomorphism $\N\xto{\sigma}C_n=\{ 1,t,\cdots,t^{n-1}\}$, from the additive
    monoid of natural numbers to the cyclic group of order $n$, given by $\sigma(1)=t$,
    is a prefibration. On the other hand,
    \[\Pcar(\sigma)=\{0,1,\cdots, n-1\}\]
    and
    \[\Car(\sigma)=\{0\}.\]
    The first equality is trivial. The second follows from part ii) of Lemma
    \ref{le.inv_sig_cart}. Thus, $\sigma$ is not a fibration.
  \end{Ex}

  \begin{Le}\label{le.prod_cart}
    Let $\sigma :M\to N$ be a monoid homomorphism.
    \begin{itemize}
      \item[i)] If $x_1$ is $\sigma$-cartesian and $x_2$ is $\sigma$-precartesian,
        then $x_1x_2$ is also $\sigma$-precartesian.
      \item[ii)] If $x_1$ and $x_2$ are $\sigma$-cartesian, then $x_1x_2$ is
        $\sigma$-cartesian.
    \end{itemize}
  \end{Le}

  \begin{proof}
    i) Set $x=x_1x_2$ and assume $\sigma(x)=\sigma(z)$. Thus $\sigma(z)=\sigma(x_1)v$,
    where $v=\sigma(x_2)$. Since $x_1$ is $\sigma$-cartesian, we can find an element
    $y_1\in M$ such that $z=x_1y_1$ and $\sigma(y_1)=v=\sigma(x_2)$. From $x_2$ being
    $\sigma$-precartesian, the existence of an $h\in\ker(\sigma)$ for which $y_1=x_2h$
    can be deduced. It follows $z=x_1x_2h$.

    We are left with having to show the uniqueness of $h$. For this, assume
    $x_1x_2h=x_1x_2h'$. Since $x_1$ is $\sigma$-cartesian, we obtain $x_2h=x_2h'$. It
    follows that $h=h'$ as $x_2$ is $\sigma$-precartesian.

    ii) Assume $x_1$ and $x_2$ are $\sigma$-cartesian elements above $u_1,u_2\in N$. We
    have to show that $x=x_1x_2$ is also $\sigma$-cartesian above $u=u_1u_2$. We have to
    check Conditions i) and ii) of Definition \ref{de.cart_sigma}. Take $z\in M$ such
    that $\sigma(z)=\sigma(x)v$. Thus $\sigma(z)=u_1u_2v$. Since $x_1$ is
    $\sigma$-cartesian, it follows that there exists a $y_1$ such that
    $\sigma(y_1)=u_2v$ and $z=x_1y_1$. Since $x_2$ is also $\sigma$-cartesian over
    $u_2$, there exists a $y_2\in M$ such that $\sigma(y_2)=v$ and $y_1=x_2y_2$. Thus
    \[z=x_1y_1=x_1x_2y_2=xy_2\]
    and Condition i) holds. Assume $xy=xy'$ and $\sigma(y)=\sigma(y')$. Thus
    $x_1x_2y=x_1x_2y'$. Since $x_1$ is $\sigma$-cartesian, we obtain $x_2y=x_2y'$, and
    by the same reasoning $y=y'$.
  \end{proof}

  \begin{Cor}
    For any monoid homomorphism $\sigma:M\to N$, the subset $\Car(\sigma)$ is a
    submonoid of $M$ which contains the subgroup $M^\sx$. Moreover, $\Pcar(\sigma)$ is
    a left $\Car(\sigma)$-subset of $M$.
  \end{Cor}

  The subset $\Pcar(\sigma)$ is not a submonoid of $M$ in general, as Example
  \ref{ex.N_to_cycl} shows.

  \begin{Le}\label{le.sig_car=pcar}
    Let $\sigma:M\to N$ be a fibration. Then
    \[\Car(\sigma)=\Pcar(\sigma).\]
  \end{Le}

  \begin{proof}
    By Corollary \ref{le.sig_cart=sig_precart}, any $\sigma$-cartesian element is
    $\sigma$-precartesian. Conversely, assume $\sigma$ is a fibration and $x\in M$ is
    $\sigma$-precartesian. Set $\sigma(x)=v$. By assumption, there exist a
    $\sigma$-cartesian $u\in M$ above $v$. Using Lemma \ref{le.sig_cart=sig_precart}
    again, we see that $u$ is $\sigma$-precartesian. Applying Lemma \ref{le.sig_pre_inv}
    gives us $x=vh$, where $h\in\ker(\sigma)$ is invertible. By part i) of Lemma
    \ref{le.inv_sig_cart}, $h$ is also $\sigma$-cartesian. Likewise, by Lemma
    \ref{le.prod_cart}, $x$ must be $\sigma$-cartesian as well.
  \end{proof}

  \begin{Pro}\label{pro.pref_prod=pref}
    Let $\sigma:M\to N$ be a prefibration. Then $\sigma$ is a fibration if and
    only if the product of two $\sigma$-precartesian elements is again
    $\sigma$-precartesian.
  \end{Pro}

  \begin{proof}
    Only the if part needs to be checked. Assume $\sigma$ is a prefibration and
    the product of two $\sigma$-cartesian elements is again $\sigma$-cartesian. It
    suffices to prove that any $\sigma$-precartesian element $x\in M$ is
    $\sigma$-cartesian. Assume $z\in M$, $v\in N$ are given such that
    $\sigma(z)=\sigma(x)v$. There exists a precartesian element $u\in M$ for which
    $\sigma(u)=v$. It follows that $\sigma(z)=\sigma(xu)$. By assumption $xu$ is
    precartesian and it follows that $z=xuh$ for an element $h\in \ker(\sigma)$. Thus,
    Condition i) of Definition \ref{de.cart_sigma} holds with $y=uh$.

    For Condition ii) we assume that $xy_1=xy_2$ and $\sigma(y_1)=\sigma(y_2).$ By
    assumption, there exists a $\sigma$-precartesian $y$ such that
    $\sigma(y)=\sigma(y_1)=\sigma(y_2).$ Then $y_1=yh_1$ and $y_2h_2$ for some elements
    $h_1,h_2\in \ker(\sigma)$. We get $xyh_1=xyh_2$. Since $xy$ is $\sigma$-precartesian
    we obtain $h_1=h_2$. It follows that $y_1=y_2$ and hence, Condition ii) of
    Definition \ref{de.cart_sigma} also holds.
  \end{proof}

  \begin{Le}
    Let $\sigma:M\to N$ be a surjective prefibration. If $A=\ker(\sigma)$ is a group,
    any element of $M$ is $\sigma$-cartesian. In particular, $\sigma$ is a fibration.
  \end{Le}

  \begin{proof}
    Our first assertion is that any element $x\in M$ is $\sigma$-precartesian. Denote
    $n=\sigma(x)$ and assume $u$ is $\sigma$-precartesian over $n$. If
    $\sigma(y)=\sigma(x)$, then $x=ua$ and $y=ub$ for some $a,b\in A$. It follows that
    $y=xc$ for $c=a^{-1}b$. For the uniqueness of $c$, assume $y=xd$, $d\in A$. We
    obtain $ub=uad$ and by the weak cancellation property, we get $b=ad$. Hence
    $d=a^{-1}b=c$, and the claim follows.

    By Proposition \ref{pro.pref_prod=pref}, $\sigma$ is a fibration. This allows us to
    assume $u$ is $\sigma$-cartesian in the previous argument. If $xy_1=xy_2$, we can
    deduce that $uay_1=uay_2$. Since $u$ is $\sigma$-cartesian, $ay_1=ay_2$ follows.
    From the invertibility of $a\in A$ we obtain $y_1=y_2$. This implies Condition ii)
    of Definition \ref{de.cart_sigma}.

    For Property i), assume $\sigma(z)=\sigma(x)v$. Since $\sigma$ is surjective, we can
    write $v=\sigma(w)$ for some $w\in M$. We now have $\sigma(z)=\sigma(xw)$ and since
    every element is $\sigma$-precartesian, $z=xwa$ for some $a\in A$. By putting
    $y=wa$, we see that Condition i) must hold as well.
  \end{proof}

  \begin{Le}
    Let $\sigma:M\to N$ be a surjective homomorphism and $A=\ker(\sigma)$. If all
    elements of $M$ are $\sigma$-precartesian, then $A$ is a group and $\sigma$ is a
    fibration.
  \end{Le}

  \begin{proof}
    Since $M=\Pcar(\sigma)$, it follows that $\sigma$ is a prefibration, and hence it is
    even fibration, thanks to Proposition \ref{pro.pref_prod=pref}. Take now $x\in M$.
    Since any element is  $\sigma$-precartesian and $\sigma(1)=\sigma(x)$, we see that
    $1=xy$ for uniquely defined $y$. By the same reason there exists $z$ such that
    $yz=1$. Since $x=x(yz)=(xy)z=z$ we see that $yx=1$ and hence $A$ is a group.
  \end{proof}

  \begin{Le}
    Let $L\xto{\rho}M$ and $M\xto{\sigma}N$ be homomorphisms of monoids.
    \begin{itemize}
      \item[i)] If $\rho$ is a fibration and $\sigma$ is a prefibration, then $\sigma
        \rho:L\to N$ is also a prefibration.
      \item[ii)] If $L\xto{\rho} M$ and $M\xto{\sigma} N$ are fibrations, then
        $L\xto{\sigma \rho} N$ is also a fibration.
    \end{itemize}
  \end{Le}

  \begin{proof}
    i) Take $n\in N$ and choose a $\sigma$-precartesian element $m\in M$ such that
    $\sigma(m)=n$. Next, choose a $\rho$-cartesian $x\in L$ above $m$. Clearly,
    $\sigma\rho(x)=n$. We wish to show that such an $x$ is $\sigma\rho$-precartesian.
    Assume $\sigma\rho(x)=\sigma\rho(z).$ Both $m=\rho(x)$ and $\rho(z)$ are above $n$.
    Since $m$ is $\sigma$-precartesian, we have $\rho(z)=mh=\rho(x)h$, where $h\in
    \ker(\sigma)$. As $x$ is $\rho$-cartesian, we can find a $u\in L$ such that $z=xu$
    and $\rho(u)=h$. It follows that $u\in \ker(\sigma\rho)$.

    It remains to show that $u$ is unique with these properties. Assume $z=xu=xu'$,
    where $\sigma\rho(u')=1$. Then $\rho(x)\rho(u')=\rho(x)\rho(u)$. This is the same as
    $m\rho(u)=m\rho(u')$. Since $m$ is $\sigma$-precartesian, we obtain
    $\rho(u)=\rho(u')$. This, together with the equality $xu=xu'$, implies $u$=$u'$,
    since $x$ is $\rho$-cartesian.

    ii) Take $n\in N$ and choose a $\sigma$-cartesian element $m\in M$ such that
    $\sigma(m)=n$. Next, choose a $\rho$-cartesian $x\in L$ such that $\rho(x)=m$. Let
    us prove that $x$ is also $\sigma\rho$-cartesian. Take $z\in L$ and $v\in N$ such
    that $\sigma\rho (z)=\sigma\rho (x)v$. Since $m=\rho(x)$ is $\sigma$-cartesian,
    there exists a $w\in M$ such that $\rho(z)=\rho(x)w$ and $\sigma(w)=v$. Using the
    fact that $x$ is $\rho$-cartesian yields an element $u\in L$ such that $z=xu$ and
    $\rho(u)=w$. It follows that $\sigma\rho(u)=v$.

    Thus, Condition i) of Definition \ref{de.cart_sigma} holds. Next, assume $xu=xu'$
    and $\sigma\rho(u)=\sigma\rho(u')$. Since $\rho(x)=m$ is cartesian, we obtain
    $\rho(u)=\rho(u')$. Since $x$ is $\rho$-cartesian, we see that $u=u'$ and Condition
    ii) of Definition \ref{de.cart_sigma} holds as well. The result follows.
  \end{proof}

  \begin{Le}
    If $M$ and $N$ are monoids, the projection $\pi:M\times N\to N$ is a fibration. In
    particular, any isomorphism is a fibration.
  \end{Le}

  \begin{proof}
    It suffice to show that for any $n\in N$, the element $x=(1,n)$ is $\pi$-cartesian.
    Take any $z\in M\times N$ and any factorisation $\pi(z)=\pi(x)v$, where $v\in N$. It
    is clear that $z$ is of the form $z=(m,nv)$. Put $y=(m,v)$. Then $z=xy$ and
    $\pi(y)=v$, showing Part i) of Definition \ref{de.cart_sigma}.

    For Part ii), take $y_1=(m_1,v_1)$ and $y_2=(m_2,v_2)$ with properties $xy_1=xy_2$
    and $\pi(y_1)=\pi_2(y)$. Then $v_1=v_2$ and $(m_1,nv_1)=(m_2,nv_2)$. So, $m_1=m_2$
    which implies $y_1=y_2$.
  \end{proof}

  We will generalise the above fact in Corollary \ref{cor.ps_fib}.

  \begin{Le}
    Let $\sigma_i: M_i\to N_i$, $i=1,2$ be fibrations (resp. prefibrations). The map
    \[\sigma=\sigma_1\times \sigma_2:M_1\times M_2\to N_1\times N_2,\]
    given by
    \[\sigma(m_1,m_2)=(\sigma_1(m_1),\sigma_2(m_2))\]
    is also a fibration (resp. prefibrations).
  \end{Le} 

  \begin{proof}
    We will only consider fibrations, as the case of prefibrations is proved in an
    analogous manner. Take $n=(n_1,n_2)\in N_1\times N_2$ and choose a
    $\sigma_i$-cartesian element $x_i\in M_i$ such that $\sigma_i(x_i)=n_i$, $i=1,2$.
    Let us prove that $x=(x_1,x_2)\in M_1\times M_2$ is $\sigma$-cartesian. Take
    $z=(z_1,z_2)\in M_1\times M_2$ and $v=(v_1,v_2)\in N_1\times N_2$ such that
    $\sigma(z)=\sigma(x)v$. We have $\sigma_i(z_i)=\sigma_i(x)v_i$ for $i=1,2$. By our
    choice, there are $u_1\in M_1$ and $u_2\in M_2$ such that $z=xu$ and
    $\sigma_i(u_i)=v_i$, $i=1,2$. Here, $u=(u_1,u_2)$. Assume $u'=(u'_1,u'_2)$ satisfies
    the conditions $z=xu'$ and $\sigma(u')=v$ as well. We obtain
    $\sigma_i(u_i)=\sigma_i(u_i')$, $i=1,2$. By our choice, this gives us $u_i=u_i'$.
  \end{proof}
	
  \begin{Le}
    Let
    \[\xymatrix{M\ar[r]^{\sigma_2} \ar[d]_{\tau_2}&N\ar[d]^{\tau_1} \\
      K\ar[r]_{\sigma_1}&L}\]
    be a pull-back diagram of monoids. If $\sigma_1$ is a fibration, then $\sigma_2$ is
    also a fibration.
  \end{Le}
	
  \begin{proof}
    It suffices to note that for any $n\in N$ and any $\sigma_1$-cartesian $k\in K$
    above $\tau_1(n)$, the element $(k,n)\in M$ is $\sigma$-cartesian above $n$. 
  \end{proof}
  \vspace{2em}

  We introduce several 2-categories associated to pre-fibrations over a monoid $N$.
  Firstly, for a monoid $N$, denote by $\fMon/N$ the following 2-category:

  Objects are homomorphism of monoids $\sigma:M\to N$. Morphisms from $\sigma:M\to N$ to
  $\sigma':M'\to N$ are homomorphisms $\alpha:M\to M'$ for which the diagram
  \[\xymatrix{M\ar[rr]^{\alpha}\ar[dr]_{\sigma} && M'\ar[dl]^{\sigma'} \\ & N &}\]
  commutes. Let $\alpha$ and $\beta$ be morphisms from $\sigma:M\to N$ to $\sigma':M'\to
  N$ in $\fMon/N$. A 2-morphism $\alpha\To\beta$ is an element $c\in M'$ for which
  $\sigma'(c)=0$, and $c\alpha(m)=\beta(m)c$ holds for all $m\in M$.
  \vspace{2em}

  Let $\fPFib_N$ be the full sub-2-category of $\fMon/N$ whose objects are prefibrations
  $\sigma:M\to N$. To introduce the next 2-category, we need the following definition:

  \begin{De}
    Let $\alpha$ be a morphism in $\fPFib_N$ from a prefibration $M\xto{\sigma}N$ to a
    prefibration $M'\xto{\sigma'}N$. Then $\alpha$ is called \emph{cartesian} if
    \[\alpha(\Pcar(\sigma))\subseteq\Pcar(\sigma').\]
  \end{De}

  Observe that if $\sigma$ and $\sigma'$ are fibrations, then $\alpha$ sends
  $\sigma$-cartesian elements to $\sigma'$-cartesians, because in this case,
  $\Pcar=\Car$.

  Finally, let $\fFib_N$ be the sub-2-category of $\fPFib_N$ whose objects are monoid
  fibrations $\sigma:M\to N$ and morphisms are cartesian morphisms.

\section{The Grothendieck construction}\label{sec.Groth_const}

  It is well-known that if $A$ is an $M$-monoid, we can define their semidirect product.
  This it simply $M\times A$ as a sit. The monoid structure is defined by
  $(m,a)(n,b)=(mn,\phi_n(a)b)$. We can generalise this using the Grothendieck
  construction when $A$ is a lax $M$-module. This new construction, which we will
  denote by $\Groth(A,\phi,\gamma)$, can be seen as a 2-categorical version of the
  semidirect product.

  \begin{Le}\label{le.Groth_const} 
    Assume there is given a lax action of a monoid $M$ on a monoid $A$. The following
    product on the cartesian product $M\times A$ defines a monoid structure.
    \[(m,a)(n,b)=(mn,\gamma_{m,n}\phi_n(a)b).\]
    We call this the Grothendieck construction.
  \end{Le}

  \begin{proof}
    Obviously, $(1,1)$ is the identity element. To show associativity, consider
    \begin{align*}
      \left((m,a)(n,b)\right)(k,c)
        &=(mn,\gamma_{m,n}\phi_n(a) b)(k,c)\\
        &=(mnk,\gamma_{mn,k}\phi_k(\gamma_{m,n}\phi_n(a) b)c)\\
        &=(mnk,\gamma_{mn,k}\phi_k(\gamma_{m,n})\phi_k(\phi_n(a))\phi_k(b)c)
    \end{align*}
    On the other hand, we have
    \begin{align*}
      (m,a)\left((n,b)(k,c)\right)&=(m,a)(nk,\gamma_{n,k}\phi_k(b)c)\\
      &=(mnk,\gamma_{m,nk}\phi_{nk}(a)\gamma_{n,k}\phi_k(b)c)
    \end{align*}
    Using first Identity ii) of Definition \ref{de.lax_act} and then Identity iii)
    allows us to rewrite the last expression as
    \[(mnk,\gamma_{m,nk}\gamma_{n,k}\phi_k\phi_n(a)\phi_k(b)c)=
      (mnk,\gamma_{mn,k}\phi_k(\gamma_{m,n}), \phi_k\phi_n(a)\phi_k(b)c).\]
    Comparing these expressions, we see that the associativity law holds.
  \end{proof}

  \begin{Rem}
    Recall Remark \ref{rem.grp_ext}. In its notation, the Grothendieck construction
    $\Groth(G,\phi,f)$ is isomorphic to $B$, thanks to \cite{homology}. For example,
    take $G=\{\pm 1\}$ to be the cyclic group of order two and $\Pi$ the Klein Vier
    group. Thus $\Pi=\{1,x,y,xy\}$, with the multiplication given by,
    \[x^2=y^2=1,\quad xy=yx.\]
    It follows that $(G,\phi,\gamma)$ is a pseudo $\Pi$-group and the corresponding
    Grothendieck construction is the quaternion group $\{\pm 1,\pm i,\pm j,\pm
    k\}$. Here, $\phi_u=\id_G$ for all $u\in \Pi$ and
    \[\gamma_{x,x}=\gamma_{x,xy}=\gamma_{y,x}=\gamma_{y,y}=\gamma_{xy,y}=
    \gamma_{xy,xy}=-1\]
    and
    \[\gamma_{x,y}=\gamma_{y,xy}=\gamma_{xy,x}=1.\]
  \end{Rem}

  As already mentioned, we denote the monoid constructed in Lemma \ref{le.Groth_const}
  by $\Groth(A,\phi,\gamma)$.

  The following holds:
  \begin{Le}\label{le.lax_surj}
    \begin{itemize}
      \item[i)] Let $(A,\phi,\gamma)$ be a monoid with a lax action of $M$. The map
        \[\Groth(A,\phi,\gamma)\xto{\sigma}M,\quad\sigma(m,a)=m\]
        is a surjective homomorphism and
        \[1\to A\xto{\iota}\Groth(A,\phi,\gamma)\xto{\sigma} M\to 1\]
        is is a short exact sequence of monoids. Here $\iota(a)=(1,a)$.
      \item[ii)] Any element of the form $(m,1)$ is $\sigma$-precartesian and hence,
        $\sigma$ is a prefibration. 
      \item[iii)] Any $\sigma$-precartesian element has the form $(m,a)$, where $a$ is
        an invertible element in $A$.
    \end{itemize}
  \end{Le}

  \begin{proof}
    i) This is clear.
  
    ii) Exactness follows since
    \[\ker(\sigma)=\{(1,a)|a\in A\}.\]
    To see that $(m,1)$ is $\sigma$-precartesian, observe that if
    $\sigma(n,b)=\sigma(m,1)$, it follows that $n=m$. Since
    \[(m,1)(1,b)=(m,\gamma_{m1}\phi_1(1)b)=(m,b),\]
    we see that any element from $\sigma^{-1}(m)$ has an expected decomposition. To see
    uniqueness, observe that $(m,b)=(m,1)(1,c)$ implies $b=c$.
  
    iii) Assume $(m,a)$ is also $\sigma$-precartesian. Since $(m,1)(1,a)=(m,a)$, we can
    use Lemma \ref{le.sig_pre_inv} to conclude that $(1,a)$ is invertible. This
    implies the result.
  \end{proof}

  \begin{Cor}\label{cor.ps_fib}
    Let $(A,\phi,\gamma)$ be a monoid with a pseudo action of $M$. The map
    \[\Groth(A,\phi,\gamma)\xto{\sigma}M,\quad\sigma(m,a)=m\]
    is a fibration.
  \end{Cor}

  \begin{proof}
    We only need to show that the product of two $\sigma$-precartesian elements $(m,a)$
    and $(m,b)$ is again $\sigma$-precartesian (see Proposition
    \ref{pro.pref_prod=pref}). By the previous Lemma \ref{le.lax_surj}, $a$ and $b$ are
    invertible elements in $A$. We also have
    \[(m,a)(n,b)=(mn,\gamma_{m,n}\phi_n(a)b)\]
    since $\gamma_{m,n}$ is invertible. We conclude that $\gamma_{m,n}\phi_n(a)b)$ is
    invertible and so, $(m,a)(n,b)$ is $\sigma$-precartesian, thanks to Part iii) of
    Lemma \ref{le.lax_surj}.
  \end{proof}

  \begin{Le}\label{le.lax}
    Let $(\alpha,\tau):(A,\phi,\gamma)\to(A',\phi',\gamma')$ be a lax $M$-homomorphism
    of lax $M$-monoids. The map
    \[\bar{\alpha}:\Groth(A,\phi,\gamma)\to\Groth(A',\phi',\gamma'),\]
    given by
    \[\bar{\alpha}(m,a)=(m,\tau_m\alpha(a)),\]
    is a monoid homomorphism and the diagram
    \[\xymatrix{\Groth(A,\phi,\gamma)\ar[rr]^{\bar{\alpha}}\ar[dr]_{\sigma} &&
    \Groth(A',\phi',\gamma')\ar[dl]^{\sigma'}\\& M&}\]
    commutes. Here, $\sigma(m,a)=m=\sigma'(m,a')$ as above.
  \end{Le}

  \begin{proof}
    The commutativity of the diagram is obvious, as is the fact that $\bar{\alpha}$
    respects unite elements. It remains to show that $\bar{\alpha}$ respects the
    product. Take two elements $(m,a)$ and $(m,b)$ in $\Groth(A,\phi,\gamma)$. We have
    \begin{align*}
      \bar{\alpha}((m,a)(n,b))
        &=\bar{\alpha}(mn,\gamma_{m,n}\phi_n(a)b)\\
        &=(mn,\tau_{mn}\alpha(\gamma_{m,n})\alpha(\phi_n(a))\alpha(b))\\
        &=(mn,\gamma'_{m,n}\phi'_n(\tau_{m})\tau_n\alpha(\phi_n(a))\alpha(b))\\
        &=(mn,\gamma'_{m,n}\phi'_n(\tau_{m})\phi'_n(\alpha(a))\tau_n\alpha(b))\\
        &=(mn,\gamma'_{m,n}\phi'_n(\tau_{m}\alpha(a))\tau_n\alpha(b))\\
        &=(m,\tau_{m}\alpha(a))(n, \tau_n\alpha(b))\\
        &=\bar{\alpha}(m,a)\bar{\alpha}(n,b).
    \end{align*}
  \end{proof}	

  \begin{Le}\label{le.lax_hom_2cell}
    Let
    \[(\alpha,\tau),(\beta,\theta):(A,\phi,\gamma)\to(A',\phi',\gamma')\]
    be a lax $M$-homomorphisms of lax $M$-monoids and $c'\in A'$ a 2-cell
    $\alpha\To\beta$. Then $(1,c')$ defines the 2-cell
   \[\bar{\alpha}\To\bar{\beta}:\Groth(A,\phi,\gamma)\to \Groth(A',\phi',\gamma').\]
  \end{Le}

  \begin{proof}
    By assumption,
    \begin{equation}\label{eq.lax_hom_2cell1}
      c'\alpha(a)=\beta(a)c'\quad {\rm and } \quad \phi_m'(c')\tau_m=\theta_m c'
    \end{equation}
    holds for every $m\in M$ and every $a\in A$. We need to show that for every
    $m\in M, a\in A$, the equality
    \begin{equation}\label{eq.lax_hom_2cell2}
      (1,c')(m,\tau_m\alpha(a))=(m,\theta_m\beta(a))(1,c')
    \end{equation}
    holds in $\Groth(A',\phi',\gamma')$. Observe that the equation
    \[(1,c')(m,\tau_m\alpha(a))=(m,\gamma'_{1,m}\phi'_m(c')\tau_m\alpha(a))\]
    holds. In the vein of (\ref{eq.lax_hom_2cell1}), we see that
    \[(1,c')(m,\tau_m\alpha(a))=(m,\theta_mc'\alpha(a))\]
    as $\gamma_{1,m}=1$. On the other hand, we have
    \[(m,\theta_m\beta(a))(1,c')=(m,\gamma'_{m,1}\phi_1'\theta_m\beta(a)c').\]
    Since $\gamma_{m,1}=1$ and $\phi_1(a')=a'$, we can rewrite this as
    \[(m,\theta_m\beta(a))(1,c')=(m,\theta_m\beta(a)c')=(m,\theta_mc'\alpha(a)).\]
    We used Identity (\ref{eq.lax_hom_2cell1}) in the last step. Comparing these computations,
    Identity (\ref{eq.lax_hom_2cell2}) follows. The Lemma is now a direct consequence of the product
    law in $\Groth(A',\phi',\gamma')$.
  \end{proof}

  The above lemma show that the Grothendieck construction yields a 2-functor
  \[\Groth_{pf}:\fLMon_M\to\fFib_N\]
  from the 2-category of lax $M$-monoids to the 2-category of monoid prefibrations. It
  restricts to the 2-functor 
  \[\Groth_f:\fPMon_M\to\fPFib_N\]
  from the 2-category of pseudo $M$-monoids to the 2-category of monoid fibrations. The
  following is Grothendieck's result \cite{sga1} restricted to monoids.

  \begin{Th}
    For any monoid $M$, the 2-functors $\Groth_{pf}$ and $\Groth_{f}$ are 2-equivalences.
  \end{Th}

  \begin{proof}
    The proof will effectively be given in section \ref{sec.cleave}, as we need to
    introduce cleavages first. The fact that these 2-functors are essentially surjective
    will be proved in Lemma \ref{le.Groth_M_iso}. The full and faithfulness is a
    consequence of Lemma \ref{le.lax_hom_2cell} and Proposition \ref{pro.clev_to_lax}
    below.
  \end{proof}

\section{Cleavages}\label{sec.cleave}

\subsection{From fibrations to lax $N$-monoids}

  Let $\sigma:M\to N$ be a prefibration of monoids. We set $A=\ker(\sigma)$. Denote by
  $A^\sx$ the set of invertible elements in $A$.

  \begin{De}[See \cite{sga1}]
    A \emph{cleavage} of a prefibration $\sigma:M\to N$ is a function $\kappa$, which
    assigns to each element $n\in N$, a $\sigma$-precartesian element
    $\kappa(n)\in\sigma^{-1}(n)$ such that $\kappa(1)=1$.
  \end{De}

  If $\sigma$ is a fibration, then $\kappa(n)$ is automatically cartesian by Lemma
  \ref{le.sig_car=pcar}.
  \newline

  \noindent
  Denote by $\Clev(\sigma)$ the set of all cleaves of $\sigma$. It is not empty by the
  axiom of choice. By the definition of a prefibration, we have the following easy but
  very important fact.

  \begin{Le}\label{le.cleav_sec-to-ker}
    Let $\kappa$ be a cleavage of a prefibration $\sigma:M\to N$. Then $\kappa$ induces
    a unique map $\xi:M\to A$, called the \emph{cocleavage induced by $\kappa$}, for
    which
    \[x=\kappa(\sigma(x))\xi(x)\]
    for any $x\in M$.
  \end{Le}

  \begin{proof}
    Take $x\in M$. Then $\kappa(\sigma(x))$ is a $\sigma$-precartesian over $\sigma(x)$.
    Since $x$ is in the preimage of $\sigma(x)$, there exits a unique $\xi(x)\in A$ with
    the property $x=\kappa(\sigma(x))\xi(x)$, as desired.
  \end{proof}

  We will use Lemma \ref{le.cleav_sec-to-ker} several times. In particular, it allows us
  to define the two functions
  \begin{eqnarray*}
    \phi:A\times N\to A & & \phi_n(a)=\xi(a\kappa(n)), \\
    \gamma:N\times N\to A & & \gamma_{m,n}=\xi(\kappa(m)\kappa(n)).
  \end{eqnarray*}
  Here $\xi$ is a cocleavage induced by $\kappa$. These functions are characterised by
  the equalities 
  \begin{equation}\label{eq.kappa_def}
    a\kappa(n)=\kappa(n)\phi_n(a)
  \end{equation}
  and 
  \begin{equation}\label{eq.gam_def}
    \kappa(m)\kappa(n)=\kappa(mn)\gamma_{mn}.
  \end{equation}

  The following was proved by Grothendieck in the general case for fibrations.

  \begin{Pro}\label{pro.pref_lax_act}
    Let $\sigma:M\to N$ be a prefibration and $\kappa\in \Clev(\sigma)$. The functions
    $\phi:A\times N\to A$ and $\gamma:N\times N\to A$, defined in (\ref{eq.kappa_def})
    and (\ref{eq.gam_def}), give rise to a lax action of $N$ on $A$. If $\sigma:M\to N$
    is a fibration, we obtain a pseudo-action of $N$ on $A$.
  \end{Pro}

  \begin{proof}
    We have to check that Conditions i)-v) of Definition \ref{de.lax_act} are satisfied.

    i) Let $n=1$. Then $a\kappa(1)=\kappa(1)\phi_(a)$. Since $\kappa$ is normalised, we
    obtain $a=\kappa_1(a)$.
    
    ii) We have
    \[a\kappa(mn)=\kappa(mn)\phi_{mn}(a)\]
    for any $a\in A$ and $m,n\in M$ by the definition of the function $\phi$. It follows
    that
    \begin{align*}
      a\kappa(m)\kappa(n)&=a\kappa(mn)\gamma_{mn} \\
                         &=\kappa(mn)\phi_{mn}(a)\gamma_{mn}.
    \end{align*}
    We also have
    \begin{align*}
      a\kappa(m)\kappa(n)=&\kappa(m)\phi_m(a)\kappa(n) \\
                          &=\kappa(m)\kappa(n)\phi_n(\phi_m(a)) \\
                          &=\kappa(mn)\gamma_{mn}\phi_n(\phi_m(a)).
    \end{align*}
    Comparing these expressions and using the weak cancellation property, we obtain
    \[\phi_{mn}(a)\gamma_{mn}=\gamma_{mn}\phi_n(\phi_m(a)).\]

    iii) We have
    \begin{align*}
      \kappa(m)\kappa(n)\kappa(k)&=\kappa(mn)\gamma_{mn}\kappa(k) \\
                                 &=\kappa(mn)\kappa(k)\phi_k(\gamma_{mn}) \\
                                 &=\kappa(mnk)\gamma_{mn,k}\phi_k(\gamma_{mn}).
    \end{align*}
    On the others hand,
    \begin{align*}
      \kappa(m)\kappa(n)\kappa(k)&=\kappa(m)\kappa(nk)\gamma_{nk} \\
      &=\kappa(mnk)\gamma_{m,nk}\gamma_{nk}.
    \end{align*}
    Comparing these expressions and using the weak cancellation property, we obtain
    \[\gamma_{mn,k}\phi_k(\gamma_{mn})=\gamma_{m,nk}\gamma_{nk}.\]
    
    iv) By definition, $\kappa(1)\kappa(n)=\kappa(n)\gamma_{1,n}$. Since $\kappa(1)=1$,
    we see that $\gamma_{1,n}=1$. This is due to the uniqueness of $\gamma$.
    
    v) We wish to show that $\phi_n:A\to A$ is a monoid homomorphism. Take $a,b\in A$.
    We have
    \[\kappa(n)\phi_n(ab)=(ab)\kappa(n)=a\kappa(n)\phi_n(b)=\kappa(n)\phi_n(a)\phi_n(b).\]
    By the uniqueness property, $\phi_n(ab)=\phi_n(a)\phi_n(b)$, which gives us the last
    condition.

    If $\sigma$ is a fibration, then $\kappa(m)\kappa(n)$ is $\sigma$-precartesian by
    Lemma \ref{le.prod_cart}. This allows us to use Lemma \ref{le.sig_pre_inv} which
    yields the invertibility of $\gamma_{mn}$.
  \end{proof}

  Thus, having fixed a cleavage $\kappa$ of a prefibration $\sigma:M\to N$, we obtain a
  lax $M$-monoid structure on $A=\ker(\sigma)$. The Grothendieck construction gives us a
  prefibration $\Groth(A,\phi,\gamma)\to N$. Our next aim is to show that these
  prefibrations are isomorphic.

  \begin{Le}\label{le.Groth_M_iso}
    The map $\bar{\alpha}:\Groth(A,\phi,\gamma)\to M$, given by
    \[\bar{\alpha}(m,a)=\kappa(m)a,\]
    is an isomorphism in $\fFib_N$. Moreover, this morphism is also cartesian.
  \end{Le}

  \begin{proof}
    The only think to check is that $\bar{\alpha}$ is a homomorphism of monoids, as the
    rest is clear. Take two elements $(m,a)$ and $(n,b)$ in $\Groth(A,\phi,\gamma)$. We
    have
    \begin{align*}
      \bar{\alpha}((m,a)(n,b))&=\bar{\alpha}(mn, \gamma_{m,n}\phi_n(a)b) \\
                              &=\kappa(mn)\gamma_{m,n}\phi_n(a)b \\
                              &=\kappa(m)\kappa(n)\phi_n(a)b \\
                              &=\kappa(m)a\kappa(b) \\
                              &=\bar{\alpha}(m,a)\bar{\alpha}(n,b).
    \end{align*}
  \end{proof}

  \begin{Pro}\label{pro.clev_to_lax}
    Let $\sigma:M\to N$ and $\sigma':M'\to N$ be prefibrations and
    $\bar{\alpha}:M\to M'$ be monoid homomorphism for which the following diagram
    commutes
    \[\xymatrix{
      M\ar[rr]^{\bar{\alpha}}\ar[dr]_{\sigma}&&M'\ar[dl]^{\sigma'}\\
      &N.&
    }\]
    Choose cleavages $\kappa$ and $\kappa'$ of $\sigma$ and $\sigma'$, and let $(A,
    \phi,\gamma)$ and $(A',\phi',\gamma')$ be their corresponding lax $N$-monoids,
    constructed in Proposition \ref{pro.pref_lax_act}. The pair
    \[(\alpha:A\to A',\tau:N\to A')\]
    is a lax morphism
    $(A,\phi,\gamma)\to(A',\phi',\gamma')$, where $\alpha$ is the restriction of
    $\bar{\alpha}$ on $A=\ker(\sigma)$ and $\tau$ is defined uniquely by
    \[\bar{\alpha}(\kappa(n))=\kappa'(n)\tau_n.\]
  \end{Pro}

  \begin{proof}
    We have $a\kappa(n)=\kappa(n)\phi_n(a)$. Apply $\bar{\alpha}$ to obtain
    \[\alpha(a)\bar{\alpha}(\kappa(n))=\bar{\alpha}(\kappa(n))\alpha(\phi_n(a)).\]
    So
    \[\alpha(a)\kappa'(n)\tau_n=\kappa'(n)\tau_n\alpha(\phi_n(a)).\]
    This implies
    \[\kappa'(n)\phi'_n(\alpha(a))\tau_n=\kappa'(n)\tau_n\alpha(\phi_n(a)).\]
    By uniqueness, we obtain
    \[\phi'_n(\alpha(a))\tau_n=\tau_n\alpha(\phi_n(a)).\]
    Next, we have $\kappa(m)\kappa(n)=\kappa(mn)\gamma_{m,n}$. Apply $\bar{\alpha}$ to
    obtain
    \[\bar{\alpha}(\kappa(m))\bar{\alpha}(\kappa(\alpha))=
    \bar{\alpha}\kappa(mn)\alpha(\gamma_{mn}).\]
    This gives us
    \[\kappa'(m)\tau_m\kappa'(n)\tau_n=\kappa'(mn)\tau_{mn}\gamma_{m,n}\]
    \[\kappa'(m)\kappa'(n)\phi'_n(\tau_m)\tau_n=\kappa'(mn)\tau_{mn}\gamma_{m,n}\]
    and thus
    \[\gamma'_{m,n}\phi'_n(\tau_m)\tau_n=\tau_{mn}\gamma_{m,n}.\]
  \end{proof}

\subsection{Dependence on cleavages}

  We wish to discuss the dependence of the pair $(\phi,\gamma)$ on the chosen
  cleavages. Let $\kappa,\tilde{\kappa}\in\Clev(\sigma)$. By Lemma
  \ref{cor.sig_cart_uniq_h}, there exists a unique $\eta(n)\in A^\sx$ for any $n$, such
  that $\tilde{\kappa}(n)=\kappa(n)\eta(n)$. Clearly $\eta(1)=1$. Denote by $\Maps_*(N,
  A^\sx)$ the group of all maps $\eta:N\to A^\sx$ with the property $\eta(1)=1$. The
  group $\Maps_*(N, A^\sx)$ acts on $\Clev(\sigma)$ from the right:
  \[(\kappa \cdot \eta)(n):=\kappa(n)\eta(n).\]
  One easily sees that $\Clev(\sigma)$ is a right $\Maps_*(N, A^\sx)$-torsor. This
  is to say, the action is transitive and free.

  \begin{Le}\label{le.cob}
    Let $\sigma:M\to N$ be a prefibration, $\kappa,\tilde{\kappa}\in\Clev$ and
    $\eta:N\to A^\sx$ be the unique map for which $\tilde{\kappa}(n)=\kappa (n)\eta(n)$.
    Consider the pair of functions $(\phi,\gamma)$ and $(\tilde{\phi},\tilde{\gamma})$
    determined by the equations (\ref{eq.kappa_def}), (\ref{eq.gam_def}) corresponding
    to the cleavages $\kappa$ and $\tilde{\kappa}$. We have
    \[\tilde{\phi}_n(a)=\eta(n)^{-1}\phi_n(a)\eta(n)\]
    and
    \[\tilde{\gamma}_{m,n}=\eta(mn)^{-1}\gamma_{m,n}\phi_n(\eta(m))\eta(n).\]
  \end{Le}

  \begin{proof}
    The equality (\ref{eq.kappa_def}), applied to $\tilde{\kappa}$, gives us $a\kappa
    (n)\eta(n)=\kappa(n)\eta(n)\tilde{\phi}_n(a)$. It follows that
    \[\kappa(n)\phi_n(a)\eta(n)=\kappa(n)\eta(n)\tilde{\phi}_n(a).\]
    The weak cancellation property gives us the first equality.

    The same reasoning also gives us
    \[\kappa(m)\eta(m)\kappa(n)\eta(n)=\kappa(mn)\eta(mn)\tilde{\gamma}_{m,n}.\]
    Since the LHS us equal to
    \[\kappa(m)\kappa(n)\phi_n(\eta(m))\eta(n)=
    \kappa(mn)\gamma_{m,n}\phi_n(\eta(m))\eta(n),\]
    we obtain $\gamma_{m,n}\phi_n(\eta(m))\eta(n)=\eta(mn)\tilde{\gamma}_{m,n}$, and the
    second formula follows.
  \end{proof}

  This gives us solid control on the lax $N$-monoid structures arising from
  prefibrations on a given monoid $N$. The next question we wish to answer pertains to
  the morphisms of prefibrations on $N$.

  Assume we have a commutative diagram
  \[\xymatrix{M\ar[rr]^{\bar{\alpha}}\ar[dr]_{\sigma}&&M'\ar[dl]^{\sigma'}\\
    &N&},\]
  where $\sigma$ and $\sigma'$ are prefibrations. Recall that by Proposition
  \ref{pro.clev_to_lax}, if we choose cleavages $\kappa$ and $\kappa'$ for $\sigma$ and
  $\sigma'$, respectively, we obtain a morphism
  $(\alpha,\tau):(A,\phi,\gamma)\to(A',\phi',\gamma')$. Assume we have chosen nother
  cleavages $\tilde{\kappa}$ and $\tilde{\kappa'}$. As already seen, there will be
  functions $\eta:N\to A^\sx$ and $\eta':N\to A^{'\sx}$ such that
  \[\tilde{\kappa}(n)=\kappa(n)\eta(n),\quad{\rm and}\quad\tilde{\kappa'}(n)
    =\kappa'(n)\eta'(n).\]
  Let us apply Proposition \ref{pro.pref_lax_act}, first to the pair $\kappa, \kappa'$,
  and then to $\tilde{\kappa}, \tilde{\kappa}'$. We obtain pairs
  \[(\alpha:A\to A',\tau:N\to A')\quad{\rm and}\quad(\tilde{\alpha}:A\to A',
    \tilde{\tau}:N\to A').\]
  These are lax morphisms $(A,\phi,\gamma)\to(A',\phi',\gamma')$ and
  $(A,\tilde{\phi},\tilde{\gamma})\to(A',\tilde{\phi}',\tilde{\gamma}')$. Recall also
  that both $\alpha$ and $\tilde{\alpha}$ are restrictions of $\bar{\alpha}$ on
  $A=\ker(\sigma)$. This implies
  \[\tilde{\alpha}=\alpha.\]
  It follows that $\tau$ and $\tilde{\tau}$ are uniquely defined by
  $\bar{\alpha}(\kappa(n))=\kappa'(n)\tau_n$ and
  $\bar{\alpha}(\tilde{\kappa}(n))=\tilde{\kappa}'(n)\tilde{\tau}_n$, respectively.

  \begin{Le}
    In the above notation, we have
    \[\tilde{\tau}(n)=\eta'(n)^{-1}\tau(n)\alpha(\eta(n)).\]
  \end{Le}

  \begin{proof}
    By definition, we have
    \[\bar{\alpha}(\tilde{\kappa}(n))=\tilde{\kappa'(n)}\tilde{\tau}(n).\]
    The RHS equals to $\kappa'(n)\eta'(n)\tilde{\tau}(n)$, while the LHS is
    $\bar{\alpha}(\kappa(n)\eta(n))=\kappa'(n)\tau_n\alpha(\eta(n))$. Comparing
    these expression and using the weak cancellation property yields the desired
    result.
  \end{proof}

\section{On automorphism groups}\label{sec.autfib}

  Let us fix a prefibration of monoids $\sigma:M\to N$ for this section. Denote the
  kernel of $\sigma$ by $A$ and choose a cleavage $\kappa:N\to M$. We have already seen
  that there are unique functions
  \[\phi:A\times N\to A \quad{\rm and}\quad \gamma:N\times N\to A\]
  satisfying the conditions
  \[a\kappa(n)=\kappa(n)\phi_n(a) \quad{\rm and }\quad
    \kappa(m)\kappa(n)=\kappa(mn)\gamma_{mn}.\]
  According to Lemma \ref{pro.pref_lax_act}, $(A,\phi, \gamma)$ is a lax action of $N$ on
  $A$. This can be used to reconstruct $\sigma :M\to A$ as the Grothendieck construction
  of $(A,\phi, \gamma)$.

  Let $\aut_A(M)$ be the set of all automorphisms of $\psi:M\to M$ such that
  $\sigma(A)=A$ and $\psi$ is cartesian. By cartesian we mean
  \[\psi(\Pcar(\sigma))\subseteq\Pcar(\sigma).\]
  For any $\psi\in\aut_A(M)$, denote by $\theta_\psi:A\to A$ (or simply $\theta$) the
  restriction of $\psi$ on $A$. We have $\theta\in\aut(A)$ by assumption.

  \begin{Le}\label{le.aut_subgroup}
    The subset $\aut_A(M)$ is a subgroup of $\aut(M)$.
  \end{Le}

  \begin{proof}
    The only point to check, is that $\psi^{-1}$ sends $\sigma$-precartesian elements to
    $\sigma$-precartesians if $\psi\in\aut_A(M)$. Take $x\in M$ and assume $\psi(x)=y$
    is $\sigma$-precartesian. One can write $x=\kappa(n)a$, where $n=\sigma(x)$ and
    $a\in A$. Then $y=\psi(\kappa(n))\theta(a)$. By assumption on $\psi$,
    $\psi(\kappa(n))$ is precartesian. Thus, Lemma \ref{le.sig_pre_inv} implies that
    $\theta(a)$ is invertible. Since $\theta$ is an automorphism, $a$ must be
    invertible. Lemma \ref{le.prod_cart} says that $x\kappa(n)a$ is
    $\sigma$-precartesian. Thus $\psi^{-1}(\Pcar(\sigma))\subset \Pcar(\sigma).$
  \end{proof} 

  \begin{Le}\label{le.aut_ext_hom}
    Let $\psi\in\aut_A(M)$ and $\theta\in\aut(A)$ be the restriction of $\psi$ on $A$.
    There exist a unique homomorphism $\eta_\psi:N\to N$ (or simply $\eta$) for which
    the diagram
    \[\xymatrix{
      0\ar[r]&A\ar[r]\ar[d]_{\theta}&M\ar[r]^{\sigma}\ar[d]^{\psi}&N\ar[r]\ar[d]^\eta&0\\
      0\ar[r]& A\ar[r] &M\ar[r]^\sigma &N\ar[r]&0
    }\]
    commutes. Moreover $\eta$ is an automorphism.
  \end{Le}

  \begin{proof}
    Let $m_1, m_2\in M$ satisfy $\sigma(m_1)=n=\sigma(m_2)$. We have to show that
    $\sigma\psi(m_1)=\sigma(\psi(m_2))$ holds as well. This will allows us to define
    $\eta(n)$ by $\sigma\psi(m_1)$.

    Since $\sigma$ is a prefibration, there are elements $a_1,a_2\in A$ such that
    $m_i=\kappa(n) a_i$, $i=1,2$. Hence
    \[\sigma \psi(m_i)=\sigma \psi(\kappa(n))\sigma\psi(a_i),\quad i=1,2.\]
    By assumption on $\psi$, we have $\psi(a_i)\in A$. Thus $\sigma\psi(a_i)=1$ and
    therefore, $\sigma\psi(m_i)=\sigma\psi(\kappa(n))$, $i=1,2$. It follows that
    $\sigma\psi(m_1)=\sigma(\psi(m_2))$. This proves the existence of $\eta$. Since
    $\sigma$ and $\psi$ are surjective, it follows that $\eta$ is also surjective. For
    injectivity, assume $\eta(n_1)=n=\eta(n_2)$. We have $\psi(\kappa(n_i))=
    \kappa(n)a_i$, $i=1,2$ where $a_i\in A$. By assumption on $\psi$, we have
    $a_i=\theta (b_i)$ for $b_i\in A$, $i=1,2$. This gives us $\psi(\kappa(n_i))=\kappa
    (n)\psi(b_i).$ Since $\psi$ is an isomorphism, it follows that
    $\kappa(n_i)=\psi^{-1}(\kappa (n))b_i$, $i=1,2$. Apply $\sigma$ to get
    $n_i=\sigma(\psi^{-1}(\kappa(n))).$ Hence $n_1=n_2$ and injectivity of $\eta$
    follows.
  \end{proof}

  \begin{Le}\label{le.map_to_invs}
    Let $\psi\in\aut_A(M)$. There exist a unique map $\xi:N\to A^\sx$ such that
    \[\psi(\kappa(n))=\kappa(\eta(n))\xi(n).\]
    Moreover the following identities hold
    \begin{itemize}
      \item[i)] $\xi(1)=1$
      \item[ii)] $\phi_{\eta(n)}(\theta(a))\xi(n)=\xi(n)\theta(\phi_n(a))$
      \item[iii)] $\gamma_{\eta(m),\eta(n)}\phi_{\eta(n)}(\xi(m))\xi(n)=
        \xi(mn)\theta(\gamma_{m,n})$
    \end{itemize}
  \end{Le}

  \begin{proof}
    Since $\sigma\psi\kappa(n)=\sigma\kappa\eta(n)$, the function $\xi:N\to A$ exists
    and is unique. Since $\psi$ is cartesian, $\xi(n)$ is invertible. Thus $\xi:N\to
    A^\sx$.

    i) The first equality follows from the fact that $\kappa(1)=1$.

    ii) For the second equality, apply $\psi$ to Equality (\ref{eq.kappa_def}) on page
    \pageref{eq.kappa_def} to get
    \[\psi(a\kappa(n))=\psi(\kappa(n)\psi_n(a)).\]
    This can be rewritten as
    \[\theta(a)\kappa(\eta(n))\xi(n)=\kappa(\eta(n))\xi(n)\theta(\phi_n(a)).\]
    Since
    \[\theta(a)\kappa(\eta(n))=\kappa(\eta(n))\phi_{\eta(n)}\theta(n),\]
    Equality ii) follows.

    iii) For the third equality, apply $\psi$ to Equality (\ref{eq.gam_def}) on page
    \pageref{eq.gam_def} to obtain
    \[\psi(\kappa(m))\psi(\kappa(n))=\psi(\kappa(mn))\theta(\gamma_{m,n}).\]
    From this, we can obtain
    \[\kappa(\eta(m))\xi(m)\kappa(\eta(n))\xi(n)=
      \kappa(\eta(mn))\xi(mn)\theta(\gamma_{m,n}).\]
    Observe that $\xi(m)\kappa(\eta(n))=\kappa(\eta(n))\phi_{\eta(n)}(\xi(n))$. Hence
    \[\kappa(\eta(m))\kappa(\eta(n))\phi_{\eta(n)}(\xi(n))\xi(n)=
      \kappa(\eta(mn))\xi(mn)\theta(\gamma_{m,n}).\]
    We use the fact that
    \[\kappa(\eta(m))\kappa(\eta(n))=\kappa(\eta(mn))\gamma_{\eta(m),\eta(n)}\]
    to obtain this last part.
  \end{proof}

  \begin{Le}\label{le.aut3}
    Assume there are given automorphisms $\eta\in\aut(N)$ and $\theta\in\aut(A)$. Let
    $\xi:N\to A^\sx$ be a function satisfying Relations i) - iii) from Lemma
    \ref{le.map_to_invs}. There exist a unique element $\psi\in\aut_A(M)$ such that
    $\theta=\theta_\psi$, $\eta=\eta_\psi$. Moreover
    \begin{equation}\label{eq.aut3}
      \psi(\kappa(n)a)=\kappa ({\eta(n)})\xi(n)\theta(a)
    \end{equation}
    holds for any $a\in A$, $n\in N$.
  \end{Le}

  \begin{proof}
    Any element of $M$ can be uniquely expressed as the product $\kappa(n)a$. This
    immediately implies the uniqueness of $\psi$. Subsequently, (\ref{eq.aut3}) gives a
    well defined map $\psi:M\to M$ whose restriction on $A$ is $\theta$ and $\eta\sigma=
    \sigma\psi$. It remains to show that $\psi$ is an automorphism and cartesian. We
    have
    \begin{align*}
      \psi(\kappa(m)a\kappa(n)b) &=\psi(\kappa(m)\kappa(n)\phi_n(a)b) \\
                                 &=\psi(\kappa(mn)\gamma_{m,n}\phi_n(a)b) \\
                                 &=\kappa_{\eta(mn)}\xi(mn)\theta(\gamma_{m,n})
                                   \theta(\phi_n(a)) \theta(b) \\
                                 &=\kappa_{\eta(mn)}\gamma_{\eta(m),\eta(n)}\phi_{\eta(n)}
                                   (\xi(m))\xi(n)\theta(\phi_n(a)) \theta(b) \\
                                 &=\kappa_{\eta(mn)}\gamma_{\eta(m),\eta(n)}\phi_{\eta(n)}
                                   (\xi(m))\phi_{\eta(n)}\theta(a)\xi(n)\theta(b) \\
                                 &=\kappa(\eta(m))\kappa(\eta(n))\phi_{\eta(n)}
                                   (\xi(m)\theta(a))\xi(n)\theta(b) \\
                                 &=\kappa(\eta(m))\xi(m)\theta(a)\kappa(\eta(n))
                                   \xi(n)\theta(b)=\psi(\kappa(m)a)\psi(\kappa(n)b).
    \end{align*}
    We have shown that $\psi$ is a homomorphism. To see the bijectivity of $\psi$,
    assume $\psi(\kappa(n)a)=\psi(\kappa(n')a')$. Then
    \[\kappa(\eta(n))\xi(n)\theta(a)=\kappa(\eta(n'))\xi(n)\theta(a).\]
    Apply $\sigma$ and use the fact that $\xi(n)\theta(a),\xi(n')\theta(a')\in A$ to
    obtain $n=n'$. Since $\kappa(n)$ is $\sigma$-precartesian, it follows that
    $\xi(n)\theta(a)=\xi(n)\theta(a')\in A$. We see that $\theta(a)=\theta(a')$ as
    $\xi(n)$ is invertible. It follows now that $a=a'$ as $\theta$ is an automorphism
    and the injectivity of $\psi$ is proved.

    For surjectivity, take any element $\kappa(m)b\in M$. Define $n=\eta^{-1}(m)$ and
    $a=\theta^{-1}(b\xi(n)^{-1})$. Then $\psi(\kappa(n)a)=\kappa(m)b$. It remains to
    show that $\psi$ is cartesian. By definition, we have the equality
    \[\psi(\kappa(n))=\kappa(\eta(n))\xi(n)\]
    and $\kappa(\eta(n))$ is $\sigma$-precartesian, $\xi(n)$ must be invertible.
    We have shown that $\psi(\kappa(n))$ is $\sigma$-precartesian.
  \end{proof}

  According to Lemmas \ref{le.aut_subgroup} and \ref{le.aut_ext_hom}, the assignment
  $\psi\mapsto(\theta_\psi,\eta_\psi)$ defines the group homomorphism
  \[\aut_A(M)\to\aut(A)\times\aut(N).\]
  It will be the aim of the rest of this section to investigate this homomorphism. We
  will need some additional notations to state our results.
  \newline

  \noindent
  Denote by $C$ the following subset of $\aut(A)\times\aut(N)$: A pair $(\theta,\eta)$
  belongs to $C$ if and only if there exists a map $\alpha:N\to A^\sx$ such that
  \[\phi_{\eta(n)}\theta(a)\alpha(n)=\alpha(n)\theta(\phi_n(a))\] for any $a\in A$ and
  $n\in N$. According to Identity ii) of Lemma \ref{le.map_to_invs}, the pair
  $(\theta_\psi,\eta_\psi)$ belongs to $C$ for any $\psi\in\aut_A(M)$. We thusly have a
  map
  \begin{equation}\label{eq.rho}
    \rho:\aut_A(M)\to C\subset\aut(A)\times\aut(N),
  \end{equation}
  where $\rho(\psi)=(\theta_\psi,\eta_\psi)$.

  \begin{Le}
    The above constructed subset $C$ is a subgroup of $\aut(A)\times\aut(N)$.
  \end{Le}

  \begin{proof}
    Let $(\theta, \eta)$ and $(\vartheta,\zeta)$ be elements in $C$. There are maps
    $\alpha,\beta:N\to A^\sx$ such that
    \[\phi_{\eta(n)}\theta(a)\alpha(n)=\alpha(n)\theta(\phi_n(a))\]
    and
    \[\phi_{\zeta(n)}\vartheta(a)\beta(n)=\beta(n)\vartheta(\phi_n(a)).\]
    We can deduce
    \begin{align*}
      \phi_{\eta\zeta(n)}\theta(\vartheta(a))\alpha(\zeta(n))\theta(\beta(n))
        &=\alpha(\zeta(n))\theta(\phi_{\zeta(n)}(\vartheta(a))\theta(\beta(n)) \\
        &=\alpha(\zeta(n))\theta(\beta(n))\theta(\vartheta\phi_n(a)).
    \end{align*}
    Since $\alpha(\zeta(n))\theta(\beta(n))\in A^\sx$, we see that $(\theta\circ
    \vartheta, \eta\circ\zeta)\in C$.
  \end{proof}

\section{The case when $A$ is commutative}\label{sec.commutative}

  Throughout this section, we assume that there is given a pseudo action of a monoid $M$
  on a commutative monoid $A$. Condition ii) of Definition \ref{de.lax_act} in this case
  simplifies to
  \[\phi_n(\phi_m(a))=\phi_{mn}(a).\]
  Subsequently, $\phi$ becomes a (real) action as the map
  \[M\to\End(A),\quad m\mapsto\phi_m\]
  is now a monoid homomorphism. The second function $\gamma:M\times M\to A$ also
  simplifies and becomes a type of $2$-cocycle condition as we have
  \[\gamma_{mn,k}\phi_k(\gamma_{m,n})=\gamma_{m,nk}\gamma_{n,k}.\]

  It is known \cite{manuell}, but perhaps not well-known, that monoid prefibrations are
  essentially the same as Schreier extensions. This can be seen by simply comparing the
  definitions \cite{redei, I, Str, schP1, schP2, APachII}. Our emphasis will be more on
  fibrations, which correspond to regular Schreier extensions \cite{PaForumMath}.

  For completeness sake, let us also recall the following definitions (see
  \cite[Definitions 3.1 and 3.7]{PaForumMath} and \cite[Definition 4.6]{APachII}).

  \begin{De}
    Let $A, N$ be monoids with $A$ additionally commutative and assume there is given an
    action of $N$ on $A$, $(a,n)\mapsto an$.
    A sequence
    \[\xymatrix{E: 1\ar[r] &A \ar[r]^\iota & M \ar[r]^\sigma & N\ar[r]&1}\]
    of monoids and monoid homomorphisms is called a Schreier extension of $N$ by
    $A$ if the following conditions hold:
    \begin{enumerate}
      \item $\sigma\iota(a)=1$ for all $a\in A$,
      \item $\sigma$ is a prefibration,
      \item $\iota$ induces an isomorphism $A\to \ker(\sigma)$ of $N$-modules,
      \item $\iota(a)x=x\iota(\phi_{\sigma(x)}(a))$ for all $a\in A$ and $x\in M.$
    \end{enumerate}
    If $\sigma$ is a fibration, the extension is called a regular Schreier extension.
	\end{De}

  Such extensions form a category $\MonExt$. Objects are Schreier extensions, while
  morphisms are commutative diagrams
  \[\xymatrix{E: 1\ar[r] & A\ar[r]^\iota\ar[d]_{\alpha} &
                M\ar[r]^{\sigma}\ar[d]^{\beta} & N\ar[r]\ar[d]^{\mu} & 1 \\
              E':	1\ar[r]&  A\ar[r]_{\iota} & M'\ar[r]^{\sigma'} &N'\ar[r] & 1.}\]
  Here, $\alpha$, $\beta$, $\mu$ are monoid homomorphisms and $\beta$ is cartesian.
  One additionally requires that
  \[\alpha(ax)=\alpha(a)\mu(x)\]
  for any $a\in A$ and $x\in N$. By cartesian we mean that
  \[\beta(\Pcar(\sigma))\subseteq\Pcar(\sigma').\]
  We say that $(\alpha,\beta, \mu): (E)\to (E')$ is a morphism of extensions in this
  case.

  Two Schreier extensions $(E)$ and $(E')$ are called \emph{congruent} if there exists a
  morphism
  \[(\id_A,\beta,\id_M)(E)\to (E').\]
  Denote the congruence classes of such extensions by $\monext(M,A)$. The subclass
  formed by regular Schreier extensions is denoted by $\regmonext(M,A)$.
  \newline

  \noindent
  Denote by $\cyc(N,A)$ the collection of all $2$-cocycles. That is to say, the
  collection of all maps $\gamma:N\times N\to A$ satisfying the conditions 
  \begin{enumerate}
    \item[i)] $\gamma_{mn,k}\phi_k(\gamma_{m,n})=\gamma_{m,nk}\gamma_{n,k}$,
    \item[ii)] $\gamma_{x,1}=1=f_{1,x}$,
  \end{enumerate}
  for all $x,y,z\in N$.

  A 2-cocycle $\gamma$ is called regular if $\gamma(x,y)\in A^\sx$ for all $x,y\in N$.
  \newline

  \noindent
  Let $\gamma$ and $\gamma'$ be two 2-cocycles. We write $\gamma\sim \gamma'$ if there
  exist a function $\tau\colon N\to A^\sx$ with $\tau(1)=0$ such that
  \begin{equation}\label{eq.coc}
    \gamma_{m,n}\tau(mn)=\gamma'_{m,n}\tau(n)\phi_n(\tau(m))
  \end{equation}
  for all $x,y\in N$. We sometimes also write $\tau:\gamma\sim \gamma'$ to specify the
  role of the function $\tau$.
  \newline

  \noindent
  Denote the equivalence classes of 2-cocycles by $H^2(N,A)$. The subclass formed by
  regular 2-cocycles is denoted by $H^2(N,A^\sx)$.
  \newline

  Now we are in position to prove the following well-known fact (see \cite{APachII} for
  part i)).

  \begin{Th}
    There are canonical bijections
    \[H^2(N,A)\to \monext(N,A),\]
    \[H^2(N,A^\sx)\to \regmonext(N,A).\]
  \end{Th}

  \begin{proof}
    Since $A$ is an $N$-module, we have an endomorphism $\phi:N\to\End(A)$, given by
    $\phi_n(a)=an$. Let $\gamma\in\cyc(N,A)$ be a cocycle. The pair $(\phi,\gamma)$
    defines a lax action of $N$ on $A$. We can consider the Grothendieck construction as
    it is defined in Lemma \ref{le.Groth_const}. Thanks to Lemma \ref{le.lax_surj} and
    part ii) of Lemma \ref{le.lax_surj}, we obtain a Schreier extension. We have defined
    a map $\cyc(N,A)\to\monext(N,A)$, which is surjective thanks to Proposition
    \ref{pro.pref_lax_act} and Lemma \ref{le.Groth_M_iso}.

    If $\gamma\sim\gamma'$ and $\tau$ satisfies Equality in \ref{eq.coc}, the pair
    $(\id_A,\tau)$ defines a pseudo $N$-homomorphism $(A,\phi,\gamma)\to
    (A,\phi,\gamma')$. Hence, Lemma \ref{le.lax} implies that it induces a congruence
    from their corresponding regular Schreier extensions. The surjection
    $\cyc(N,A)\to\monext(N,A)$ thus factors through the homomorphism ${H}^2(N,A)\to
    \monext(N,A)$. Moreover, if $\gamma$ and $\gamma'$ are two $2$-cocycles such that
    the corresponding extensions are similar, Proposition \ref{pro.clev_to_lax} implies
    that $\gamma\sim\gamma'$.

    Let $\gamma\in\cyc(N,A^\sx)$. The pair $(\phi,\gamma)$ defines a pseudo action. By
    Corollary \ref{cor.ps_fib}, the corresponding Schreier extension is regular and
    therefore, the result follows.
  \end{proof}

\subsection{On automorphism groups of extensions}

  We will still consider fibrations $\sigma:M\to N$ in this section for which
  $A=\ker(\sigma)$ is commutative. We can thus assume that
  \[\xymatrix{E: 1\ar[r] &A \ar[r]^\iota & M \ar[r]^\sigma & N\ar[r]&1}\]
  is a regular Schreier extension.

  Let $\kappa:N\to M$ be a cleavage. We have already seen that this yields an action
  $\phi$ of $N$ on $A$ and a 2-cocycle $\gamma:N\times N\to A$ at the start of Section
  \ref{sec.commutative}. Since $\sigma$ is a fibration, $\gamma$ has values in $A^\sx$.
  Subsequently, it belongs to the group $\cyc(N,A^\sx)$.

  Recall that $\aut_A(M)$ denotes the group of automorphisms $\psi:M\to M$ for which
  $\psi(A)=A$ and $\psi$ is cartesian. As already mentioned, we have a
  commutative diagram
  \[\xymatrix{0\ar[r] & A\ar[r] \ar[d]_{\theta_\psi} & M\ar[r]^{\sigma}\ar[d]^{\psi} &
                N\ar[r] \ar[d]^{\eta_\psi} & 0 \\
              0\ar[r]& A\ar[r] &M\ar[r]^\sigma &N\ar[r]	&0}\]
  in this setting. Consider the following subgroups:
  \begin{eqnarray*}
    \aut^{A,N} &=& \{\psi\in \aut_A(M)| \theta_\psi=\id_A, \eta_\psi=\id_N\} \\
    \aut^A(M)  &=& \{\psi\in \aut_A(M)| \theta_\psi=\id_A\} \\
    \aut^N_A(M)&=& \{\psi\in \aut_A(M)| \eta_\psi=\id_N\}.
  \end{eqnarray*}
  Recall that the homomorphism $\rho:\aut_A(M)\to C$ (\ref{eq.rho}) defined on page
  \pageref{eq.rho}). We now introduce the following subgroups of $C$:
  \begin{eqnarray*}
    C_1 &=& \{\theta\in Aut(A)|(\theta,\id_N)\in C\},\\
    C_2 &=& \{\eta\in Aut(N)|(\id_A,\eta)\in C\}.
  \end{eqnarray*}
  One easily sees that $\theta\in C_1$ if and only if $\theta:A\to A$ respects the
  action of $N$. That is to say, if $\theta(an)=\theta(a)n$ holds for all $n\in N$ and
  $a\in A$. Similarly, $\eta\in C_2$ if and only if $ax=a\eta(x)$ for all $x\in N$ and
  $a\in A$. The homomorphism $\rho$ yields homomorphisms
  \[\rho_1:Aut^N_A\to C_1 \quad {\rm and} \quad \rho_2:Aut^A(M)\to C_2,\]
  given respectively by
  \[\rho_1(\psi)=\theta_\psi \quad {\rm and} \quad \rho_2(\psi)=\eta_\psi.\]
  \newline

  We are now in a position to prove the following result, which was previously proved
  for groups \cite[Theorem 1]{yadav}.

  \begin{Th}\label{th.main_7}
    The following hold:
    \begin{itemize}
      \item[i)] For any 1-cocycle $\xi\in {\sf Z}^1(N,A^\sx)$, the map $x\mapsto
        x\iota(\delta(\sigma x))$ is an automorphism of $M$ which belongs to
        $Aut^{A,N}$. In this way, one obtains an isomorphism
        \[{\sf Z}^1(N,A^\sx)\to Aut^{A,N}.\]
      \item[ii)] For any $\theta\in C_1$, the map
        \[(m,n)\mapsto\gamma_{m,n}(\theta\gamma_{m,n})^{-1}\]
        is a 2-cocycle. The induced map
        \[C_1\xto{\lambda_1} H^2(N,A^\sx)\]
      	is independent of the chosen cleavage $\kappa$. 
      \item[iii)] One has an exact sequence
        \[\xymatrix{
          1\ar[r]& Aut^{A,N}\ar[r] &Aut^N_A\ar[r]^{\rho_1}&C_1\ar[r]^{\lambda_1}\,\ar[r]&
            H^2(N,A^\sx).
        }\]
      \item[iv)] For any $\eta\in C_2$, the association
        \[(m,n)\mapsto\gamma_{\eta(m),\eta(n)}\gamma_{m,n}^{-1}\]
        defines a 2-cocycle and the induced map
        \[\lambda_2\colon C_2\to H^2(N,A^\sx)\]
        is independent on the choice of the cleavage $\kappa$.
      \item[v)] We have an exact sequence:
        \[1\to Aut^{A,N}(M)\to Aut^A(M)\xto{\rho_2} C_2\xto{\lambda_2} H^2(N,A^\sx).\]
    \end{itemize}
  \end{Th}

  \begin{proof}
    i) By definition, $\psi\in Aut^{A,N}$ if and only if $\theta=\id_A$ and
    $\eta=\id_N$. Hence by Lemmas \ref{le.map_to_invs} and \ref{le.aut3}, the automorphism $\psi$
    is uniquely defined by a map $\xi:N\to A^\sx$, which must satisfy Conditions i) -
    iii) of Lemma \ref{le.map_to_invs}. These conditions directly imply that $\xi$ is a 1-cocycle
    since $A$ is commutative and the values of $\gamma$ are invertible. This finishes
    the proof of this part.

    ii) Since $\gamma$ is a 2-cocycle and $\theta$ is a morphism of $N$-modules,
    $\theta(\gamma)$ must also be a 2-cocycle. Thus, $(m,n)\mapsto\gamma_{m,n}
    (\theta\gamma_{m,n})^{-1}$ is a 2-cocycle as well since $\cyc(N,A^\sx)$ is a group.
    Let $\tilde{\kappa}$ be another cleavage, $\tilde{\phi}$ the corresponding action
    and $\tilde{\gamma}$ and associated 2-cocycle. It follows from Lemma \ref{le.cob}
    that $\phi=\tilde{\phi}$ and $\gamma\sim\tilde{\gamma}$ since $A$ is commutative.
    This also implies that $\theta(\gamma)\sim \theta(\tilde{\gamma})$ and the result
    follows. 

    iii) If we compare the definitions, we see that $\aut^{A,N}$ is the kernel of
    $\rho_1$, which is the exactness at $\aut^N_A$. To check the exactness at $C_1$,
    take $\theta\in C_1$ and assume $\lambda_1(\theta)=0$. Thus the 2-cocycle
    $(m,n)\mapsto \gamma_{m,n}(\theta\gamma_{m,n})^{-1}$ is a coboundary. An other way
    of expressing that is the existence of a function $\xi\colon N\to A^\sx$ with
    $\xi(1)=0$ such that
    \[\xi(mn)\theta(\gamma_{m,n})=\gamma_{m,n}\phi(n)(\xi(m))\xi(n).\]
    We can now use Lemma \ref{le.aut3} to construct an element $\psi\in Aut^N_A$ for
    which $\rho_1(\psi)=\theta$.

    iv) To show that the map $(m,n)\mapsto \gamma_{\eta(m),\eta(n)}\gamma_{m,n}^{-1}$
    belongs to $\cyc(N,A^\sx)$ it suffice to show that $\gamma'_{m,n}\in\cyc(N,A^\sx)$.
    Here, $\gamma'_{m,n}=\gamma_{\eta(m),\eta(n)}$. In fact, we have
    \[\gamma_{m,nk}'\gamma_{n,k}'=\gamma_{\eta(m),\eta(nk)}\gamma_{\eta(n),\eta(k)}=
      \gamma_{\eta(mn),\eta(k)}\phi_{\eta(k)}(\gamma_{\eta(m),\gamma(n}).\]
    Since $\eta\in C_2$, we have $\phi_{\eta(k)}(\gamma_{\eta(m),\gamma(n})=
    \phi_{k}(\gamma_{\eta(m),\gamma(n})$. The above expression thus equals to
    \[\gamma_{m,nk}'\gamma_{n,k}'=
      \gamma_{\eta(mn),\eta(k)}\phi_{k}(\gamma_{\eta(m),\eta(n})=
      \gamma'_{mn,k} \phi_k(\gamma'_{m,n}).\]
    It follows that $\gamma'_{m,n}\in\cyc(N,A^\sx)$.

    Let $\tilde{\kappa}$ be another cleavage. We can use Lemma \ref{le.cob} to conclude
    the existence of a map $\xi:N\to A^\sx$ such that
    \[\xi:\gamma\sim \tilde{\gamma},\]
    where $\tilde{\gamma}$ is the corresponding 2-cocycle. The fact that $\eta\in C_2$
    implies that $\xi':\gamma'\sim \tilde{\gamma'}$, where $\xi'(n)=\xi(\eta(n))$. The
    proof of this part follows.

    v) The definition of $\aut^{A,N}$ once again shows that it is the kernel of
    $\rho_2$, and exactness at $\aut^A(M)$ follows. We will now check the exactness at
    $C_2$. Take $\eta\in C_2$ and assume $\lambda_2(\eta)=0$. There exist a function
    $\xi\colon N\to A^\sx$ such that $\xi(1)=0$ and
    \[\xi(mn)\gamma_{\eta(m),\eta(n)})=\gamma_{m,n}\phi(n)(\xi(m))\xi(n).\]
    We can use Lemma \ref{le.aut3} to construct an element $\psi\in Aut^N_A$ for which
    $\rho_2(\psi)=\eta$.
  \end{proof}

\end{document}